\documentclass[12pt, a4paper]{article}
\usepackage{amsmath}
\usepackage{amsmath}
\usepackage{mathtools}
\usepackage{amsfonts}
\usepackage{graphicx}
\usepackage{amsthm,amssymb}
\usepackage{epsfig}
\usepackage{float}
\usepackage{tikz} 
\usepackage{color}
\usepackage{gastex}
\usepackage[top=1in, bottom=1in, left=1in, right=1in]{geometry}
\usepackage{float}
\usepackage{changepage}
\floatstyle{plaintop}
\restylefloat{table}

\numberwithin{figure}{section}

\newtheorem{lemma}{Lemma}[section]

\newtheorem{remark}{Remark}[section]

\begin{document}
	
	\title{Generalized Grötzsch Graphs \footnote{Email: 
			upadhyay@bhu.ac.in (A. K.  Upadhyay).}}
	\author{Ashish Upadhyay\\ Department of Mathematics\\BHU, Varanasi - 221005}

\date{}
	\maketitle

	\begin{abstract} The aim of this paper is to present a generalization of Gr\"otzsch graph. Inspired by structure of the Gr\"otzsch's graph, we present constructions of two families of graphs, $G_m$ and $H_m$  for odd and even values of $m$ respectively and on $n = 2m +1$ vertices. We show that each member of this family is non-planar, triangle-free, and Hamiltonian. Further, when $m$ is odd the graph $G_m$ is maximal triangle-free, and when $m$ is even, the addition of exactly $\frac{m}{2}$ edges makes the graph $H_m$ maximal triangle-free. We show that $G_m$ is 4-chromatic and $H_m$ is 3-chromatic for all $m$. Further, we note some other properties of these graphs and compare with Mycielski's construction.
	\end{abstract}

\section{Introduction}

Named after  Herbert Grötzsch, the Gr\"otzsch graph was used to show that the planar triangle-free graphs are 3 colorable in 1959. Since then there have been several works related to exploring and extending the properties of this graph (see, for example, \cite{Thomassen},\cite{LaLuSt} \cite{GimThom}, \cite{Youngs}). The number of vertices of the Gr\"otzsch graph is 11 which can be decomposed as 5 + 5 + 1 depending on three classes of vertices that constitute this graph, $i.e$ for this graph $11 = 2\times 5 + 1$ or $11 = 2\times m + 1$ with $m = 5$. This lead to a curious little question as to if the values of $m$ are varied over positive integers $ \geq 5$ would it be possible to construct similar graphs with properties comparable to that of the Gr\"otzsch graph? In pursuance of an answer to this curiosity, it was possible to construct two families of graphs denoted by $G_m$ and $H_m$ with $2m +1$ vertices for $m \geq 5$ odd and even respectively. The graphs $G_m$ are exactly what can be termed as generalized Gr\"otzsch graphs as they possess all the properties of this graph viz. chromatic number, being maximal triangle-free, being Hamiltonian, embeddable on the projective plane, etc. The construction of graphs $H_m$ are along similar lines except that they marginally slip to possess some properties of the Gr\"otzsch graph viz. being maximal triangle-free (although they are triangle-free), chromatic number lesser by one, embeddability, etc. 

In what follows, Section 2 contains preliminaries and required definitions followed by Section 3 containing examples, definitions, and construction of the graphs $G_m$ and $H_m$. The next section, Section 4, contains the results that we have mentioned above along with their proofs.

\section{Preliminary}
The graphs considered in this paper are finite and simple. The standard graph theoretic terms used in the paper are same as those given in \cite{Wilson}. A path $v_1 e_1 v_2 e_2 \ldots v_k e_k v_{k+1}$ in a graph is a finite sequence of vertices and edges such that $v_j$ and $v_{j+1}$ are the end vertices of the edge $e_j$. A cycle in $G$ is a path for which the initial and end vertices coincide. A graph is said to be triangle-free if it does not contain any cycle of length three. The graph is called maximal triangle-free if the addition of an edge to it introduces a triangle in it. The cycle which passes through every vertex of a graph exactly once is called a Hamiltonian cycle and the graph is called Hamiltonian in this case.  A proper vertex coloring of a graph $G$ is an assignment of colors to the vertices of $G$ such that adjacent vertices receive different colors. The smallest number of colors needed for proper vertex coloring of a graph is called its chromatic number. A graph $G$ is said to be planar if it can be drawn on the plane without any edge crossing.  The smallest number of edge crossings in a planar drawing of $G$ is called its crossing number. 

\section{The Graphs: Construction and Example}

Let $n \geq 11$ be an odd integer, $n = 2m + 1$. Define the following to be elements forming the set of vertices  and edges  of a graph $G_m$ and $H_m$ as indicated:,

\begin{itemize}

\item The vertex set consists of the vertices labelled $a_n$, $p_1, p_2, \ldots, p_{m}$, $q_1, q_2, \ldots, q_{m}$

\item The edge set consists of $a_{n}q_{i}$ where $1 \leq i \leq  m$ and 

For the following description of edges  we take $q_0 = q_m$.  Furthermore, when the indices are strictly greater than $m$ they are to be taken $modulo m$.

\item $p_{i}p_{i+1}$ where $1 \leq i \leq  m$  together with 

\item following additional edges for $G_m$ : the edges are $p_{i}q_{i+(2k-1)}$ for $1\leq i \leq m$ and $ 0 \leq k  \leq  \frac{(m - 3 )}{2} $, whereas

\item following additional edges for $H_m$ : the edges are $p_{i}q_{i+(2k-1)}$ for $1\leq i \leq m$ and $ 0 \leq k  \leq  \frac{(m - 2 )}{2} $  

\end{itemize}

We present the graphs pictorially for the values $m = 5, 6$ and 7.


\begin{adjustwidth}{-3cm}{1cm}
\begin{tikzpicture}[node distance={15mm}, thick]
\draw (12, 4) -- (10, 3) -- (10, 1) -- (12, 0) -- (14, 1) -- (14, 3) -- cycle;
\node  at (12, 3.95) {\small $\bullet$};
\node  at (12, 4.2) {\small $p_1$};
\node  at (10, 2.95) {\small $\bullet$};
\node  at (10, 3.2) {\small $p_2$};
\node  at (10, 1) {\small $\bullet$};
\node  at (9.7, 1) {\small $p_3$};
\node  at (12, 0) {\small $\bullet$};
\node  at (11.8, -0.2) {\small $p_4$};
\node  at (14, 1) {\small $\bullet$};
\node  at (14, 0.75) {\small $p_5$};
\node  at (14, 3) {\small $\bullet$};
\node  at (14.3, 3) {\small $p_6$};

\draw [red,thick]   (12, 4) to[out=170,in=250, distance=6cm ] (12, 1);
\draw [green,thick]   (10, 3) to[out=200,in=280, distance=6cm ] (13, 1.5);
\draw [blue,thick]   (10, 1) to[out=250,in=350, distance=6cm ] (13, 2.5);
\draw [red,thick]   (12, 0) to[out=300,in=40, distance=6cm ] (12, 3);
\draw [green,thick]   (14, 1) to[out=30,in=135, distance=6cm ] (11, 2.5);
\draw [blue,thick]   (14, 3) to[out=110,in=180, distance=6cm ] (11, 1.5);

\node  at (12, 2) {\small $\bullet$};
\node  at (12.4, 1.95) {\small $a_6$};
\node  at (12, 1) {\small $\bullet$};
\node  at (12.2, 0.8) {\small $q_4$};
\node  at (12, 3) {\small $\bullet$};
\node  at (12, 3.2) {\small $q_1$};
\node  at (11, 2.5) {\small $\bullet$};
\node  at (10.75, 2.5) {\small $q_2$};
\node  at (11, 1.5) {\small $\bullet$};
\node  at (10.9, 1.3) {\small $q_3$};
\node  at (13, 2.5) {\small $\bullet$};
\node  at (13.2, 2.8) {\small $q_6$};
\node  at (13, 1.5) {\small $\bullet$};
\node  at (13.3, 1.5) {\small $q_5$};
\draw (12, 2) -- (12, 1);
\draw (12, 2) -- (12, 3);
\draw (12, 2) -- (11, 2.5);
\draw (12, 2) -- (11, 1.5);
\draw (12, 2) -- (13, 2.5);
\draw (12, 2) -- (13, 1.5);
\draw (12, 4) -- (11, 2.5);
\draw (12, 4) -- (13, 2.5);
\draw (10, 3) -- (12, 3);
\draw (10, 3) -- (11, 1.5);
\draw (10, 1) -- (11, 2.5);
\draw (10, 1) -- (12, 1);
\draw (12, 0) -- (11, 1.5);
\draw (12, 0) -- (13, 1.5);
\draw (14, 1) -- (12, 1);
\draw (14, 1) -- (13, 2.5);
\draw (14, 3) -- (13, 1.5);
\draw (14, 3) -- (12, 3);

\draw [red,thick]   (2, 4) to[out=170,in=220, distance=5cm ] (1.5, 1.3);
\draw [green,thick]   (0, 3) to[out=200,in=280, distance=5cm ] (2.5, 1.3);
\draw [blue,thick]   (0, 1.5) to[out=250,in=320, distance=5cm ] (3, 1.7);
\draw [red,thick]   (1, 0.5) to[out=300,in=10, distance=5cm ] (2.7, 2.5);
\draw [green,thick]   (3, 0.5) to[out=340,in=45, distance=5cm ] (2, 2.8);
\draw [blue,thick]   (4, 1.5) to[out=45,in=120, distance=5cm ] (1.3, 2.5);
\draw [orange,thick]   (4, 3) to[out=100,in=170, distance=5cm ] (1, 1.7);

\draw (2, 2) -- (2, 2.8);
\draw (2, 2) -- (2.7, 2.5);
\draw (2, 2) -- (1.3, 2.5);
\draw (2, 2) -- (1.5, 1.3);
\draw (2, 2) -- (1, 1.7);
\draw (2, 2) -- (2.5, 1.3);
\draw (2, 2) -- (3, 1.7);
\node  at (2, 2) {\small $\bullet$};
\node  at (2.5, 2) {\small $a_{13}$};
\node  at (2, 2.8) {\small $\bullet$};
\node  at (2, 3.1) {\small $q_1$};
\node  at (2.7, 2.5) {\small $\bullet$};
\node  at (2.9, 2.7) {\small $q_7$};
\node  at (1.3, 2.5) {\small $\bullet$};
\node  at (1, 2.5) {\small $q_2$};
\node  at (1.5, 1.3) {\small $\bullet$};
\node  at (1.5, 1.0) {\small $q_4$};
\node  at (1, 1.7) {\small $\bullet$};
\node  at (0.8, 1.5) {\small $q_3$};
\node  at (2.5, 1.3) {\small $\bullet$};
\node  at (2.8, 1.0) {\small $q_5$};
\node  at (3, 1.7) {\small $\bullet$};
\node  at (3.3, 1.7) {\small $q_6$};
\draw (2, 4) -- (0, 3) -- (0, 1.5) -- (1, 0.5) -- (3, 0.5) -- (4, 1.5) -- (4, 3) -- cycle;
\node  at (2, 3.95) {\small $\bullet$};
\node  at (2, 4.2) {\small $p_1$};
\node  at (0, 3) {\small $\bullet$};
\node  at (0, 3.2) {\small $p_2$};
\node  at (0, 1.5) {\small $\bullet$};
\node  at (-0.3, 1.4) {\small $p_3$};
\node  at (1, 0.5) {\small $\bullet$};
\node  at (0.8, 0.3) {\small $p_4$};
\node  at (3, 0.5) {\small $\bullet$};
\node  at (3, 0.3) {\small $p_5$};
\node  at (4, 1.5) {\small $\bullet$};
\node  at (4.2, 1.3) {\small $p_6$};
\node  at (4, 3) {\small $\bullet$};
\node  at (4.2, 3.2) {\small $p_7$};
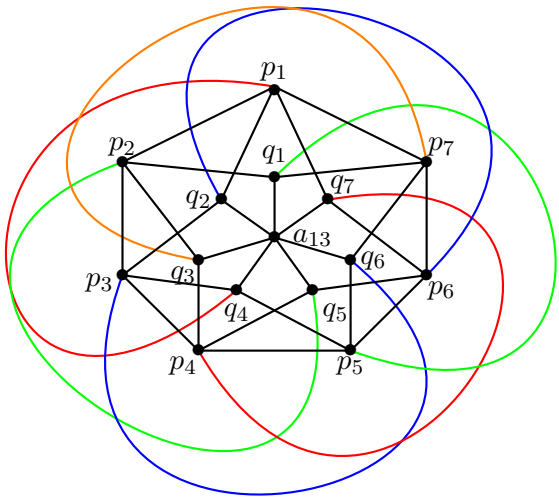
\node at (2, -2.5) {Figure 2: {\bf The graph $G_7$}};
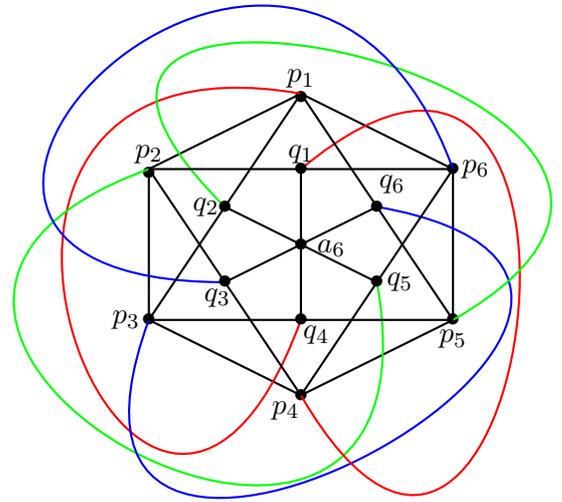
\node at (12, -2.5) {Figure 3: {\bf The graph $H_6$}};
\draw (4, 3) -- (2, 2.8);
\draw (4, 3) -- (3, 1.7);
\draw (2, 4) -- (2.7, 2.5);
\draw (2, 4) -- (1.3, 2.5);
\draw (0, 1.5) -- (1.5, 1.3);
\draw (0, 1.5) -- (1.3, 2.5);
\draw (0, 3) -- (1, 1.7);
\draw (0, 3) -- (2, 2.8);
\draw (1, 0.5) -- (2.5, 1.3);
\draw (1, 0.5) -- (1, 1.7);
\draw (3, 0.5) -- (1.5, 1.3);
\draw (3, 0.5) -- (3, 1.7);
\draw (4, 1.5) -- (2.7, 2.5);
\draw (4, 1.5) -- (2.5, 1.3);

\draw (6, 12) --  (4, 10)  -- (5, 8) -- (7, 8) -- (8, 10) --cycle;
\node  at (6, 12) {\small $\bullet$};
\node  at (6, 12.2) {\small $p_1$};
\node  at (4, 10) {\small $\bullet$};
\node  at (3.7, 10) {\small $p_2$};
\node  at (5, 8) {\small $\bullet$};
\node  at (4.7, 8) {\small $p_3$};
\node  at (7, 8) {\small $\bullet$};
\node  at (7.3, 8) {\small $p_4$};
\node  at (8, 10) {\small $\bullet$};
\node  at (8.3, 10) {\small $p_5$};
\draw (6, 9.5) -- (6, 10.5);
\draw (6, 9.5) -- (5, 10);
\draw (6, 9.5) -- (7, 10);
\draw (6, 9.5) -- (5.5, 8.5);
\draw (6, 9.5) -- (6.5, 8.5);
\draw (6, 12) -- (5, 10); \draw(5, 8) -- (5, 10);
\draw (6, 12) -- (7, 10); \draw (5, 8) -- (6.5, 8.5);
\draw (4, 10) -- (6, 10.5); \draw (7, 8) -- (5.5, 8.5);
\draw (4, 10) -- (5.5, 8.5); \draw (7, 8) -- (7, 10);
\draw (8, 10) -- (6.5, 8.5); \draw (8, 10) -- (6, 10.5);

\node at (6, 9.5) {\small $\bullet$};
\node at (6.4, 9.3) {\small $a_{11}$};
\node  at (6, 10.5) {\small $\bullet$};
\node  at (6, 10.8) {\small $q_1$};
\node  at (5, 10) {\small $\bullet$};
\node  at (4.8, 9.8) {\small $q_2$};
\node  at (5.5, 8.5) {\small $\bullet$};
\node  at (5.3, 8.4) {\small $q_3$};
\node  at (6.5, 8.5) {\small $\bullet$};
\node  at (6.8, 8.4) {\small $q_4$};
\node  at (7, 10) {\small $\bullet$};
\node  at (7.3, 10) {\small $q_5$};
\node at (6, 7) {Figure 1: {\bf The Gr\"otzsch Graph}};
 \end{tikzpicture}
\end{adjustwidth}

\section{Results}

In this section, we present some properties of the graphs $G_m$ and $H_m$ as lemma and also prove them

\begin{lemma}\label{L1}
The chromatic number of $H_m$ is 3  and that of $G_m$ is 4.

\end{lemma}

\begin{proof}
When $m$ is even assign color 1 to $a_n$ and color $2$ to $q_i$ for all $1 \leq i \leq m$. Then the vertices of the cycle $C_{m}(p_1, p_2, \ldots, p_{m})$ can  be assigned color no 1 and 3 giving rise to a proper vertex coloring of $H_m$ with the smallest number of colors in this case. The case when $m$ is odd is similar except that the cycle $C_{m}(p_1, p_2, \ldots, p_{m})$  will need two additional colors giving rise to a proper vertex coloring of $G_m$ with the smallest number of colors.
\end{proof}

\begin{lemma}\label{L2}
The graphs $G_m$ and $H_m$ are both Hamiltonian.
\end{lemma}

\begin{proof}
We give a Hamiltonian cycle in the graphs. Starting with $p_1$ we follow the path $p_1 q_2 p_3 q_4 \ldots$ $ p_{m-1} q_m a_n q_1 p_m$, where $n = 2m +1$. This is a Hamiltonian cycle in $G_m$ and $H_m$ both.

\end{proof}

\begin{lemma}\label{L3} 
The graphs $G_m$ and $H_m$ are triangle-free.  

\end{lemma}

\begin{proof} Let  $m$ be as above, It is clear that there are no triangles containing the vertex $a_n$. By symmetry, it is sufficient to consider the cycles containing $p_1$. It is also immediate that there is no triangle consisting entirely of vertices of the type $p_i$. Since $q_i$'s are not mutually connected by edges so the existence of a triangle containing two vertices of the type $q_i$ is ruled out. Thus the only possibility is that a triangle, if it exists, is of the type $p_1q_ip_j$ for some $i$ and $j$. It is easy to see that the values of $j$ are $2$ and $n$. For both these values of $j$ no value of $i$ is possible such that $p_1$ and $q_i$ form an edge. Hence it follows that the graph $G_m$ and $H_m$ are both triangle-free.

\end{proof}

\begin{lemma}\label{L4}
For $m$ odd, the graphs $G_m$ are maximal triangle-free. 
\end{lemma}
	
\begin{proof}
It is clear that the addition of an edge of type $p_i{a_n}$ or $q_iq_j$ introduces a triangle. By symmetry, it is sufficient to consider the analysis at vertex $p_1$. If we add a new edge of type $p_1q_k$ it is easy to see that $q_k$ forms an edge with $p_2$ or $p_n$. Leading to the formation of a triangle in the graph $G_m$.

On the other hand, introducing a new edge of the type $p_1p_k$ introduces either a triangle $p_1p_2p_3$, $p_1p_{m-1}p_m$ or of the following type: $p_1$ $p_k$. For $k$ odd have a $q_k$ in common for some $k$ and for $k$ even  $p_k q_{m}$ is always an edge which in turn forms a triangle with $p_1$. In other words, the star of $q_m$ consists of vertices $p_1$, $a_n$, $p_{m-1}$, $p_{m-3}$, \ldots, $p_{m-5}$.

\end{proof}

\begin{remark}
It is easy to see that the graphs  $H_m$ can be made maximal triangle-free by adding $\frac{m}{2}$ number of edges of the type $p_ip_{i+ \frac{m}{2} }$.
\end{remark}

\begin{remark}(Embeddings)
In \cite{Zeps}, the author has given an embedding of the Gr\"otzsch graph on the projective plane. Following the same pattern of vertices with a slight change in labeling it is easy to show that the graphs $G_m$ admit embedding on the projective plane. Similarly, it can be easily seen that the graphs $H_m$ embed on a pinched Torus with pinching at the vertex $a_n$. 
\end{remark}

\begin{remark}(Girth)
Although the graphs $G_m$ and $H_m$ are both triangle-free, it is easy to see that both these graphs have girth 4. 

\end{remark}

\begin{remark} (Mycielski's Construction)

This was remarked in \cite{West}. The black edges in $G_{m}$ comprise the result of applying Mycielski's construction to the cycle $C_m$, let us call it $M_m$.  We then add edges to $M_m$ which results in $G_m$.   The graph $M_m$ already has chromatic number 4 and is triangle-free. Thus $G_m$ is a maximal triangle-free supergraph of $M_m$ which preserves the chromatic number. 

\end{remark}

\section{Conclusion}
All the graphs $G_m$ as constructed in this article are non-planar triangle-free and 4 chromatic. They possess all of those properties of the Gr\"otzsch graph which do not change with an increase in the number of vertices or edges. Thus the graphs $G_m$ may be rightly named as the generalized Gr\"otzsch graphs. The graphs $H_m$ are companion graphs constructed following the same procedures as $G_m$'s. These graphs may serve as test cases for results in graph theory as much remains to be explored about them.

\section*{Acknowledgement}
The work of the author is partially supported by a grant from SERB DST through the project (MTR/2020/000006).


\begin{thebibliography}{99}

  \bibitem{Thomassen} Thomassen C., A short list color proof of Grotzsch’s theorem, {\em Journal of Combinatorial Theory}, Series B 88 (2003) 189–192



 \bibitem{GimThom}  Gimbel J. and Thomassen C., Coloring graphs with fixed genus and girth, {\em Trans. of the Amer. Math. Soc.},
Volume 349, Number 11, November 1997, Pages 4555 -- 4564

\bibitem{LaLuSt} La H., Lužar B., Štorgel K., Further extensions of the Grötzsch Theorem, {\em Discrete Mathematics,} Volume 345, Issue 6, 2022, 112849,

\bibitem{West} West D.; Personal Communication.

\bibitem{Wilson} Wilson R. J., Introduction to Graph Theory, 4th Ed, Pearson, 2012.

\bibitem{Youngs} Youngs D. A., 4-Charomatic Projective Graphs, {\em Journal of Graph Theory}, Vol. 21, No. 2, 219-227 (1996)
\bibitem{Zeps} Zeps D., On building 4-critical plane and projective planemultiwheels from odd wheels, {\em https://arxiv.org/abs/1202.4862}

  \end{thebibliography}
  \end{document}